\documentclass{article}
\usepackage[latin1]{inputenc}
\usepackage[T1]{fontenc}
\usepackage{amsmath,amssymb}
\usepackage{amsfonts,amsthm}
\usepackage{bbm}
\usepackage{mathrsfs}

\usepackage{geometry}
\usepackage{color}
   \definecolor{cites}{rgb}{0.75 , 0.00 , 0.00}  % colour for citations
   \definecolor{urls} {rgb}{0.00 , 0.00 , 1.00}  % colour for URL's
   \definecolor{links}{rgb}{0.00 , 0.00 , 0.5}   % colour for links
  \definecolor{gray}{rgb}{0.5,0.5,.5}
\usepackage{graphicx}
\usepackage{cancel}
\usepackage[plainpages=false,colorlinks=true, citecolor=cites, urlcolor=urls, linkcolor=links, breaklinks=true]{hyperref}

\allowdisplaybreaks

\newcommand{\C}{\mathbb{C}}

\newcommand{\N}{\mathbb{N}}

\newcommand{\1}{\mathbbm{1}}

\newcommand{\Ac}{\mathcal{A}}
\newcommand{\Bc}{\mathcal{B}}
\newcommand{\Kc}{\mathcal{K}}
\newcommand{\Lc}{\mathcal{L}}
\newcommand{\Pc}{\mathcal{P}}

\newcommand{\Tc}{\mathcal{T}}

\renewcommand{\epsilon}{\varepsilon}

\newcommand{\vertiii}[1]{{\left\vert\kern-0.25ex\left\vert\kern-0.25ex\left\vert #1
    \right\vert\kern-0.25ex\right\vert\kern-0.25ex\right\vert}}
\newcommand{\vertiiis}[1]{{\vert\kern-0.25ex\vert\kern-0.25ex\vert #1
    \vert\kern-0.25ex\vert\kern-0.25ex\vert}}
    
\newcommand\pto{
   \unitlength0.1ex
   \begin{picture}(30,15)
   \put(15,16){\makebox(0,0)[]{\tiny $\Pc$}}
   \put(15,5){\makebox(0,0)[]{$\to$}}
   \end{picture}
}

\newcommand\starto{
   \unitlength0.1ex
   \begin{picture}(30,15)
   \put(15,16){\makebox(0,0)[]{$*$}}
   \put(15,5){\makebox(0,0)[]{$\to$}}
   \end{picture}
}

\DeclareMathOperator{\BUC}{BUC}

\DeclareMathOperator{\diam}{diam}
\DeclareMathOperator{\dist}{dist}

\DeclareMathOperator{\Osc}{Osc}
\DeclareMathOperator*{\slim}{s-\lim}

\DeclareMathOperator{\supp}{supp}
\DeclareMathOperator{\BMO}{BMO}

\DeclareMathOperator{\VO}{VO}

\DeclareMathOperator{\EssCom}{EssCom}
\DeclareMathOperator{\BDO}{BDO}
\DeclareMathOperator{\BO}{BO}

\DeclareMathOperator{\Imag}{Im}

\newcommand{\from}{\colon}

\providecommand{\scpr}[2]{\left\langle #1, #2 \right\rangle}
\renewcommand{\sp}{\scpr}
\providecommand{\abs}[1]{\left\lvert#1\right\rvert}
\providecommand{\norm}[1]{\left\lVert#1\right\rVert}
\providecommand{\set}[1]{\left\{ #1\right\}}

\newtheorem{thm}{Theorem}
\newtheorem*{thm*}{Theorem}

\newtheorem{lem}[thm]{Lemma}
\newtheorem{prop}[thm]{Proposition}
\newtheorem{cor}[thm]{Corollary}
\newtheorem*{cor*}{Corollary}

\theoremstyle{definition}

\theoremstyle{remark}

\numberwithin{equation}{section}

\allowdisplaybreaks
\begin{document}
\title{\bf Essential Commutants and Characterizations of the Toeplitz Algebra}
%\date{}
\author{Raffael Hagger\footnote{The author has received funding from the European Union's Horizon 2020 research and innovation programme under the Marie Sk{\l}odowska-Curie grant agreement No 844451.\newline Department of Mathematics and Statistics, University of Reading, Whiteknights Campus, Reading RG6 6AX, United Kingdom, r.t.hagger@reading.ac.uk}}
\maketitle
\vspace{-0.4cm}
\begin{abstract}
In this paper we study the Toeplitz algebra, which is generated by Toeplitz operators with bounded symbols on the Fock space $F^p_{\alpha}$. We show that the Toeplitz algebra coincides with each of the algebras generated by band-dominated, sufficiently localized and weakly localized operators, respectively. Moreover, we determine its essential commutant and its essential bicommutant. For $p = 2$ these results were obtained recently by Xia. However, Xia's ideas are mostly connected to Hilbert space theory and methods which are not applicable for $p \neq 2$. Instead, we use a recent result of Fulsche to generalize Xia's theorems.


\end{abstract}

\section{Introduction} \label{sec:Introduction}

Even though the Toeplitz algebras on Bergman and Fock spaces are probably among the most studied operator algebras, it is a notoriously difficult problem to determine whether a given operator actually belongs to the Toeplitz algebra. Until very recently there were no satisfactory characterizations and several authors came up with (seemingly) larger algebras that are easier to work with. The most prominent examples are the algebras of band-dominated \cite{BaFu,FuHa,HaggerUnitBall,HaggerBSD,HaSe,HaVi,WuZhe}, sufficiently localized \cite{BaFu,Isralowitz,IsMiWi,XiZhe} and weakly localized operators \cite{BaFu,IsMiWi,SaZo,WuZhe,Xia2015,Xia2018}. All of these algebras include the Toeplitz algebra, but it was unknown whether they actually contain operators outside of the Toeplitz algebra. In 2015, Xia \cite{Xia2015} proved the surprising result that for $p = 2$ the $C^*$-algebra generated by weakly localized operators is equal to the closure of the set of Toeplitz operators. In other words, every operator in that $C^*$-algebra can be approximated by Toeplitz operators. No products are needed. As a consequence, all of the above mentioned algebras are actually the same (see \cite{BaFu}). This result not only provides multiple workable characterizations, but also allows to switch between different viewpoints when dealing with operators in the Toeplitz algebra.

In a subsequent paper, Xia \cite{Xia2018} showed an additional characterization using essential commutants. Xia was mainly interested in the Bergman space, but the same arguments also apply to the Fock space (see \cite{WuZhe}). In the present paper we generalize all the characterizations of the Toeplitz algebra on the Fock space to $p \neq 2$. As the arguments in \cite{WuZhe,Xia2015,Xia2018} mostly depend on Hilbert space and $C^*$-algebra techniques, we have to follow a different route here. We first prove the essential commutant characterization for band-dominated operators and then show, by using a recent result by Fulsche \cite{Fulsche}, that all the above mentioned algebras again coincide. In particular, the Toeplitz algebra can be characterized by essential commutants just like in the Hilbert space case. 

To explain this in a bit more detail, let us start with some standard definitions. Let $\C^n$ denote the usual complex coordinate space of dimension $n$, which is equipped with the Euclidean dot product $\cdot$ and the corresponding absolute value $|\cdot|$:
\[z \cdot \overline{w} := z_1\overline{w}_1 + \ldots + z_n\overline{w}_n, \quad |z| := \sqrt{z \cdot \overline{z}}.\]
For $z \in \C^n$ and $r > 0$ we use $B(z,r) := \set{w \in \C^n : |z-w| < r}$ for the open ball of radius $r$ around $z$. The characteristic function of a set $S \subseteq \C^n$ will be denoted by $\1_S$. In case $S = \C^n$, we will simply write $\1 := \1_{\C^n}$.

Let $\mathrm{d}z$ denote the Lesbesgue volume form on $\C^n$ and $\alpha > 0$. The Gaussian measure $\mu_{\alpha}$ is defined by
\[\mathrm{d}\mu_{\alpha}(z) = \left(\frac{\alpha}{\pi}\right)^n e^{-\alpha|z|^2} \, \mathrm{d}z.\]
For $p \in (1,\infty)$ we will use the notation $L^p_{\alpha} := L^p(\C^n,\mu_{\alpha p/2})$ for the usual $L^p$-space defined by the measure $\mu_{\alpha p/2}$. In other words, $f \in L^p_{\alpha}$ if and only if
\[\norm{f}_p := \left[\left(\frac{\alpha p}{2\pi}\right)^n \int_{\C^n} \abs{f(z)e^{-\frac{\alpha}{2}|z|^2}}^p \, \mathrm{d}z\right]^{\frac{1}{p}}\]
is finite. We will write $f \in L^1(\C^n)$ if $f$ is integrable with respect to the Lebesgue measure and $f \in L^{\infty}(\C^n)$ if $f$ is measurable and bounded. For $f \in L^{\infty}(\C^n)$ we will use $M_f$ to denote the corresponding multiplication operator on $L^p_{\alpha}$. The sets of bounded and compact operators acting on a Banach space $X$ will be denoted by $\Lc(X)$ and $\Kc(X)$, respectively.

The Fock space $F^p_{\alpha}$ is defined as the closed subspace of entire functions in $L^p_{\alpha}$. The inner product on the ambient space $L^2_{\alpha}$ and its restriction to $F^2_{\alpha}$ will be denoted by $\sp{\cdot}{\cdot}$. As usual, $\sp{\cdot}{\cdot}$ extends to a sesquilinear form on $F^p_{\alpha} \times F^q_{\alpha}$, where $\frac{1}{p} + \frac{1}{q} = 1$. The induced map $f \mapsto \sp{\cdot}{f}$ is an (antilinear) isomorphism between $F^q_{\alpha}$ and $(F^p_{\alpha})^*$. Similarly, of course, $(L^p_{\alpha})^* \cong L^q_{\alpha}$.

The orthogonal projection $P \from L^2_{\alpha} \to F^2_{\alpha}$ is given by
\begin{equation} \label{eq:projection}
[Pf](z) = \int_{\C^n} f(w)e^{\alpha z \cdot \overline{w}} \, \mathrm{d}\mu_{\alpha}(w).
\end{equation}
The integral operator in \eqref{eq:projection} also defines a bounded projection from $L^p_{\alpha}$ onto $F^p_{\alpha}$, which we will again denote by $P$. From \eqref{eq:projection} it easily follows that $PM_{\1_K}$ and $M_{\1_K}P$ are compact for all compact sets $K \subset \C^n$ (see \cite[Proposition 7]{FuHa}, for example). For $f \in L^{\infty}$ we define the Toeplitz operator $T_f \from F^p_{\alpha} \to F^p_{\alpha}$ by $T_fg := P(fg)$. It is clear that $T_f$ is a bounded operator with $\norm{T_f} \leq \norm{P}\norm{f}_{\infty}$. The Banach algebra generated by all Toeplitz operators with bounded symbols will be denoted by $\Tc^p$.

Let us now define the generators of the algebras we want to compare with $\Tc^p$. For $A \in \Lc(L^p_{\alpha})$ the number
\[\sup\set{\dist(K,K') : K,K' \subseteq \C^n, M_{\1_{K'}}AM_{\1_K} \neq 0} \in [0,\infty]\]
is called the \emph{propagation} or \emph{band-width} of $A$. Here, $\dist(K,K') := \inf\limits_{x \in K, \, y \in K'} |x-y|$ denotes the minimal distance between $K$ and $K'$. Operators of finite propagation are called \emph{band operators} and the set of such operators is denoted by $\BO$. The closure of $\BO$ in the operator norm topology is called $\BDO^p$ and its elements are called \emph{band-dominated}. We will write $A \in \Ac^p$ and call $A \in \Lc(F^p_{\alpha})$ again \emph{band-dominated} if $AP \in \BDO^p$ \cite{FuHa,HaSe}.

An operator $A \in \Lc(F^p_{\alpha})$ is called \emph{sufficiently localized} \cite{XiZhe} if there are constants $C > 0$, $\beta > 2n$ such that
\[\sp{Ak_z}{k_w} \leq \frac{C}{(1+|z-w|)^{\beta}}\]
for all $w,z \in \C^n$. We write $\Ac_{sl}$ for the set of sufficiently localized operators and $\overline{\Ac_{sl}}$ for its closure in the operator norm.

If $A \in \Lc(F^p_{\alpha})$ and $A^* \in \Lc(F^q_{\alpha})$ both satisfy the two conditions
\begin{align*}
\sup\limits_{z \in \C^n} &\int_{\C^n} \abs{\sp{Tk_z}{k_w}} \, \mathrm{d}w < \infty,\\
\lim\limits_{r \to \infty} \sup\limits_{z \in \C^n} &\int_{\C^n \setminus B(z,r)} \abs{\sp{Tk_z}{k_w}} \, \mathrm{d}w = 0,
\end{align*}
we call $A$ \emph{weakly localized} \cite{IsMiWi}. The set of weakly localized operators and its closure are denoted by $\Ac_{wl}$ and $\overline{\Ac_{wl}}$, respectively.

The main results of this paper can now be summarized as follows (see Theorem \ref{thm:all_algebras_coincide} below for details):

\begin{thm*}~
\begin{itemize}
	\item[(i)] $\Tc^p = \Ac^p = \overline{\Ac_{sl}} = \overline{\Ac_{wl}}$.
	\item[(ii)] $A \in \Tc^p$ if and only if $[T_f,A] \in \Kc(F^p_{\alpha})$ for all $f \in \VO_b(\C^n)$.
\end{itemize}
\end{thm*}

Here, $[A,T_f] := AT_f-T_fA$ denotes the commutator of the two operators and $\VO_b(\C^n)$ stands for the set of bounded functions with \emph{vanishing oscillation} at infinity.

The paper is organized as follows. In Section \ref{sec:Preliminaries} we recall some important concepts and provide preliminary results. They are then applied in Section \ref{sec:Essential_Commutants}, where the essential commutant and the essential bicommutant of $\Ac^p$ are determined. In Section \ref{sec:Characterizations} we give a proof of our main theorem and mention a few direct corollaries. We are only concerned with the Fock space in this paper, but the results in Section \ref{sec:Essential_Commutants} are likely to hold verbatim for other spaces, too. The characterization of the Toeplitz algebra in Section \ref{sec:Characterizations}, however, requires a recent result by Fulsche \cite{Fulsche}, which so far is only available for the Fock space.

\section{Preliminaries} \label{sec:Preliminaries}

In this section we introduce the main concepts of this paper. Most of these concepts are well-known, but since they originate from different areas, we recall some important results that are needed later on.

\subsection{The Berezin transform}

Equation \eqref{eq:projection} implies that $F^2_{\alpha}$ is a reproducing kernel Hilbert space with reproducing kernel $K(z,w) = e^{\alpha z \cdot \overline{w}}$. A direct computation shows $\norm{K(\cdot,w)}_p = e^{\frac{\alpha}{2} |w|^2}$ so that the normalized reproducing kernels $k_w$, given by
\[k_w(z) := e^{\alpha z \cdot \overline{w} - \frac{\alpha}{2}|w|^2},\]
satisfy $\norm{k_w}_p = 1$ for all $w \in \C^n$ and $p \in (1,\infty)$. With this in mind we can define the Berezin transform $\Bc(A)$ of a bounded linear operator $A \in \Lc(F^p_{\alpha})$ as
\[[\Bc(A)](w) := \sp{Ak_w}{k_w}.\]
H\"older's inequality implies that $\Bc(A) \from \C^n \to \C$ is bounded and continuous. In fact, \cite[Theorem 2]{Coburn} shows that $\Bc(A)$ is Lipschitz continuous for $p = 2$, $\alpha = \frac{1}{2}$. It is probably well known that this generalizes to all $p \in (1,\infty)$ and $\alpha > 0$, but due to the lack of a suitable reference, we give a quick elementary proof of this fact.

\begin{prop} \label{prop:Berezin_Lipschitz}
There is a constant $C > 0$ (depending only on $p$ and $n$) such that for all $A \in \Lc(F^p_{\alpha})$ and $z,w \in \C^n$:
\[\abs{[\Bc(A)](z) - [\Bc(A)](w)} \leq C\sqrt{\alpha}\norm{A}|z-w|.\]
\end{prop}

To simplify the computation a bit, we introduce the so-called Weyl operators, which will also be used again later. For every $z \in \C^n$ and $f \in L^p_{\alpha}$ we define
\[W_zf := (f \circ \tau_z) \cdot k_z,\]
where $\tau_z(w) := w-z$. $W_z$ is a surjective isometry for all $z \in \C^n$ and $p \in (1,\infty)$. The inverse and the adjoint of $W_z$ are both equal to $W_{-z}$, where we interpret the adjoint as an operator on $L^q_{\alpha}$ via the usual duality pairing as explained above. As $P$ and $W_z$ commute, the same is true for $W_z$ if restricted to the Fock space $F^p_{\alpha}$. Therefore, the Berezin transform can be written as
\[[\Bc(A)](w) = \sp{W_{-w}AW_w\1}{\1} = [\Bc(W_{-w}AW_w)](0).\]
Similarly, using the product formula $W_zW_w = e^{-i\alpha\Imag(z \cdot \overline{w})}W_{z+w}$, we obtain
\[[\Bc(A)](z) = \sp{Ak_z}{k_z} = \sp{W_{-w}AW_wk_{z-w}}{k_{z-w}} = [\Bc(W_{-w}AW_w)](z-w).\]
It therefore suffices to prove Proposition \ref{prop:Berezin_Lipschitz} for $w = 0$.

\begin{proof}[Proof of Proposition \ref{prop:Berezin_Lipschitz}]
By H\"older's inequality and the discussion above, it suffices to show that $\norm{k_z - \1}_p \leq C\sqrt{\alpha}|z|$. Moreover, since $\norm{k_z}_p = 1$ for all $z \in \C^n$, we only need to consider a neighborhood of $0$, say $z \in B(0,\frac{1}{\sqrt{\alpha}})$. We have
\[k_z(w) - 1 = e^{\alpha w \cdot \overline{z} - \frac{\alpha}{2}|z|^2} - 1 = (e^{\alpha w \cdot \overline{z}} - 1)e^{-\frac{\alpha}{2}|z|^2} + e^{-\frac{\alpha}{2}|z|^2} - 1.\]
Obviously, $|e^{-\frac{\alpha}{2}|z|^2} - 1| \leq \sqrt{\alpha}|z|$ and $|e^{-\frac{\alpha}{2}|z|^2}| \leq 1$ for all $z \in \C^n$. Moreover,
\[\norm{e^{\alpha w \cdot \overline{z}} - 1}_p = \norm{\sum\limits_{j = 1}^{\infty} \frac{(\alpha w \cdot \overline{z})^j}{j!}}_p \leq \sum\limits_{j = 1}^{\infty} \frac{\norm{(\alpha w \cdot \overline{z})^j}_{p}}{j!} \leq \sum\limits_{j = 1}^{\infty} \alpha^j|z|^j\frac{\norm{|w|^j}_p}{j!},\]
where
\[\norm{|w|^j}_p = \left(\frac{\alpha p}{2}\right)^{-\frac{j}{2}}\left(\frac{\Gamma(\frac{jp}{2}+n)}{\Gamma(n)}\right)^{\frac{1}{p}} \leq \tilde{C}^j\alpha^{-\frac{j}{2}}\Gamma\left(\frac{j+1}{2}\right)\]
for some constant $\tilde{C}$ that depends on $p$ and $n$. In particular,
\[\norm{e^{\alpha w \cdot \overline{z}} - 1}_p \leq \sqrt{a}|z|\sum\limits_{j = 0}^{\infty} \alpha^{\frac{j}{2}}|z|^j\tilde{C}^{j+1}\frac{\Gamma(\frac{j}{2}+1)}{(j+1)!} \leq \sqrt{a}|z|\sum\limits_{j = 0}^{\infty} \tilde{C}^{j+1}\frac{\Gamma(\frac{j}{2}+1)}{(j+1)!}\]
for $z \in B(0,\frac{1}{\sqrt{\alpha}})$. As this series is convergent, the result follows.
\end{proof}

\subsection{Band-dominated operators}

Let $p \in (1,\infty)$. For every $t > 0$ we can choose a family of cutoff functions $\set{\varphi_{j,t} : j \in \N}$ that satisfies the following properties:
\begin{itemize}
	\item[(i)] $\sum\limits_{j = 1}^{\infty} [\varphi_{j,t}(z)]^p = 1$ for every $z \in \C^n$.
	\item[(ii)] $\sup\limits_{j \in \N} \diam(\supp \varphi_{j,t}) < \infty$.
	\item[(iii)] For $w,z \in \C^n$ with $|z-w| \leq \frac{1}{t}$ we have $\sum\limits_{j = 1}^{\infty} |\varphi_{j,t}(z) - \varphi_{j,t}(w)|^p < t$.
	\item[(iv)] For all $z \in \C^n$ and $r > 0$ the set $\set{j \in \N : \supp\varphi_{j,t} \cap B(z,r) \neq 0}$ is finite.
\end{itemize}
Here, $\supp$ denotes the (closed) support of a function and $\diam$ denotes the Euclidean diameter of a set. Such families always exist even in a much more general context (see \cite[Lemma 3.1]{HaSe}). For $\C^n$ these cutoff functions are easily constructed by hand \cite{BaIs,FuHa} and it is readily seen that these families can be chosen in such a way that there is a universal constant $N$ (depending only on the dimension $n$) satisfying
\begin{itemize}
	\item[(v)] For every $z \in \C^n$, $t > 0$ the set $\set{j \in \N : \varphi_{j,t}(z) \neq 0}$ has at most $N$ elements.
\end{itemize}
Despite the constructions in \cite{BaIs,FuHa}, where $N$ grows exponentially in $n$, the best possible constant is actually $N = 2n+1$ (see \cite[Proposition 2.2.7]{BeDra}). In $n = 1$, for example, this can be achieved by choosing hexagons instead of squares for the covering. Anyway, in this paper we only need the existence of such a constant and therefore stick with $N$. Likewise, the functions $\varphi_{j,t}$ are merely auxiliary and it is completely irrelevant which ones we choose. We therefore just assume that we chose them here satisfying (i) to (v) and fix them for the rest of the paper. The reason why these cutoff functions are useful (and also why the choice does not matter) is the following.

\begin{prop} \label{prop:BDO_characterization}
(\cite[Proposition 3.5]{HaSe})\\
Let $A \in \Lc(L^p_{\alpha})$. Then $A \in \BDO^p$ if and only if
\begin{equation} \label{eq:BDO_characterization}
\lim\limits_{t \to 0} \sup\limits_{\norm{f}_p = 1} \sum\limits_{j = 1}^{\infty} \norm{[A,M_{\varphi_{j,t}}]f}_p^p = 0.
\end{equation}
\end{prop}

This characterization is useful in many ways as it allows to commute band-dominated operators with cutoff functions for a low price. For example, it is easily shown that compact operators satisfy \eqref{eq:BDO_characterization}. Moreover, the inverse closedness of $\BDO^p$ is immediate as well. Combining these facts, Proposition \ref{prop:BDO_characterization} can be used to show the inverse closedness of the corresponding Calkin algebra $\BDO^p / \Kc(L^p_{\alpha})$ \cite[Theorem 3.7]{HaSe}, that is, $\BDO^p$ is closed with respect to Fredholm inverses.

Recall that $\Ac^p$ is the restriction of $\BDO^p$ to $F^p_{\alpha}$. To be precise, $A \in \Ac^p$ if and only if $AP \in \BDO^p$ by definition. $\Ac^p$ therefore inherits a lot of properties from $\BDO^p$.

\begin{prop} \label{prop:A^p_properties}
(\cite[Theorem 3.10]{HaSe})\\
$\Ac^p$ is an inverse closed Banach algebra which contains $\Kc(F^p_{\alpha})$ and $\Tc^p$. The corresponding Calkin algebra $\Ac^p / \Kc(F^p_{\alpha})$ is inverse closed, too.
\end{prop}

From the definition of band-dominated operators, it also easily follows that $A \in \Ac^p$ if and only if $A^* \in \Ac^q$, where $\frac{1}{p} + \frac{1}{q} = 1$.

\subsection{Limit operators}

Let $\beta\C^n$ denote the Stone-\v{C}ech compactification of $\C^n$. By the universal property of $\beta\C^n$, every continuous function $f \from \C^n \to K$ to a compact Hausforff space $K$ can be uniquely extended to $\beta\C^n$. In the following, we will not distinguish between a function and its extension. For $A \in \Ac^p$ consider the function $\Psi \from \C^n \to \Lc(F^p_{\alpha})$, $\Psi(z) = W_{-z}AW_z$. As bounded sets are relatively compact in the weak operator topology, $\Psi$ has a weakly continuous extension to $\beta\C^n$. It turns out that this extension is even strongly continuous.

\begin{prop} \label{prop:limit_operators}
(\cite[Proposition 5.3]{BaIs}, see also \cite[Theorem 4.11]{HaSe} and \cite[Lemma 4]{HaVi})\\
Let $A \in \Ac^p$. The map $\Psi \from \C^n \to \Lc(F^p_{\alpha})$, $\Psi(z) = W_{-z}AW_z$ extends to a strongly continuous map on $\beta\C^n$.
\end{prop}

Most of the time, Proposition \ref{prop:limit_operators} is used in the following form. Let $x \in \beta\C^n \setminus \C^n$ and let $(z_{\gamma})$ be a net in $\C^n$ that converges to $x$. Then the strong limit
\begin{equation} \label{eq:strong_limit}
A_x := \slim\limits_{z_{\gamma} \to x} W_{-z_{\gamma}}AW_{z_{\gamma}}
\end{equation}
exists and does not depend on the net $(z_{\gamma})$. The operators $A_x$ are called the limit operators of $A$. We will say that a net $(A_{\gamma})$ converges $*$-strongly to $A$, $A_{\gamma} \starto A$ in short, if both $(A_{\gamma}) \to A$ and $(A_{\gamma}^*) \to A^*$ strongly. In particular, the convergence in \eqref{eq:strong_limit} is $*$-strong and $A_x^* = (A^*)_x$.

A direct computation shows $W_{-z}M_fW_z = M_{f(\cdot+z)} = M_{f \circ \tau_{-z}}$ for $z \in \C^n$. Therefore, as the Weyl operators commute with $P$, we get
\[(T_f)_x = \slim\limits_{z_{\gamma} \to x} T_{f(\cdot+z_{\gamma})}\]
if we apply \eqref{eq:strong_limit} to a Toeplitz operator $T_f$. In particular, we can see that the limit operators of $T_f$ only depend on the values of $f$ close to infinity.

The most important feature of limit operators is the following.

\begin{prop} \label{prop:limit_operator_characterization}
(\cite[Theorem 1.1, Lemma 6.1]{BaIs}, \cite[Theorem 28]{FuHa}, see also \cite[Corollary 4.24, Theorem 4.38]{HaSe})\\
Let $A \in \Ac^p$. Then
\begin{itemize}
	\item[(a)] $A$ is compact if and only if $A_x = 0$ for all $x \in \beta\C^n \setminus \C^n$.
	\item[(b)] $A$ is Fredholm if and only if $A_x$ is invertible for all $x \in \beta\C^n \setminus \C^n$.
\end{itemize}
\end{prop}

There are two things to note here. First of all, in \cite{BaIs,FuHa} a different compactification (and sign convention) was used. As already noted in \cite[Section 5]{HaSe} and \cite[Remark 3]{HaVi}, this is not an issue because in any case the closure of the set $\set{W_{-z}AW_z : z \in \C^n}$ in the strong operator topology is considered. We could even replace the nets by sequences because bounded sets are metrizable in the strong operator topology. We will use this fact in Proposition \ref{prop:sequential_convergence} and Theorem \ref{thm:EssComBDO} below. However, as the points in $\beta\C^n \setminus \C^n$ cannot be reached by sequences, this comes with additional (mainly notational) difficulties. We therefore stick to nets and use the Stone-\v{C}ech compactification for simplicity.

The other thing to note is that \cite{BaIs,FuHa} only consider the Toeplitz algebra $\Tc^p$, which could possibly be smaller than $\Ac^p$. However, one of the main results of this paper is that actually $\Ac^p = \Tc^p$ for all $p \in (1,\infty)$, see Theorem \ref{thm:all_algebras_coincide} below. To avoid a possibly circular argument, we also refer to \cite{HaSe}, where Proposition \ref{prop:limit_operator_characterization} was shown (in a much more general context) for $A \in \Ac^p$ directly (also using the Stone-\v{C}ech compactification for that matter).

We conclude this section with the following well-known fact, which is particularly useful for us because it turns the strong convergence in \eqref{eq:strong_limit} into norm convergence when multiplied with a compact operator. In \cite{RaRoSi} this is proven for sequences, but the same proof also works for bounded nets.

\begin{prop} \label{prop:compact_convergence}
(\cite[Theorem 1.1.3]{RaRoSi})\\
Let $E$ be a Banach space, $A \in \Lc(E)$ and $(A_{\gamma})$ a bounded net in $\Lc(E)$. If $(A_{\gamma})$ converges strongly to $A$, then $\norm{A_{\gamma}K-AK} \to 0$ for all $K \in \Kc(E)$. Similarly, if $(A_{\gamma}^*)$ converges strongly to $A^*$, then $\norm{KA_{\gamma}-KA} \to 0$ for all $K \in \Kc(E)$.
\end{prop}

\subsection{\texorpdfstring{$\Pc$}{P}-theory}

When working on the ambient space $L^p_{\alpha}$, it will prove useful to generalize the notions of compactness and strong convergence. The following notions originate from the theory of Banach space valued sequence spaces (see \cite{Lindner,RaRoSi}, for example). The $\Pc$ stands for \emph{projection} and is referring to the projections $M_{\1_{B(0,r)}}$, $r > 0$. An operator $A \in \Lc(L^p_{\alpha})$ is called $\Pc$-compact if
\[\lim\limits_{r \to \infty} \norm{M_{\1_{B(0,r)}}A-A} = \lim\limits_{r \to \infty} \norm{AM_{\1_{B(0,r)}}-A} = 0.\]
Equivalently, $A$ is $\Pc$-compact if and only if
\[\lim\limits_{r \to \infty} \norm{M_{\1_{B(0,r)}}AM_{\1_{B(0,r)}} - A} = 0.\]
The connection between compactness and $\Pc$-compactness is readily seen:

\begin{prop} \label{prop:compact_P_compact}
Every compact operator is $\Pc$-compact. On the other hand, if $A \in \Lc(L^p_{\alpha})$ is $\Pc$-compact, then $AP$ and $PA$ are compact.
\end{prop}

\begin{proof}
Let $A \in \Kc(L^p_{\alpha})$. $M_{\1_{B(0,r)}}$ converges $*$-strongly to the identity for $r \to \infty$ and therefore the $\Pc$-compactness of $A$ follows from Proposition \ref{prop:compact_convergence}. Conversely, if $A$ is $\Pc$-compact, then $\norm{AM_{\1_{B(0,r)}}P-AP} \to 0$ as $r \to \infty$. That is, $AP$ is approximated by the compact operators $AM_{\1_{B(0,r)}}P$ and is therefore compact itself. Similarly, $PA$ is compact.
\end{proof}

Replacing compact by $\Pc$-compact operators, we can also define the $\Pc$-essential norm:
\[\norm{A}_{\Pc} := \inf\limits_K \norm{A+K},\]
where we take the infimum over all $\Pc$-compact operators $K$. In the same way we can amend the notion of ($*$-)strong convergence. We say that a bounded sequence of operators $(A_n)_{n \in \N}$ converges $\Pc$-strongly to $A \in \Lc(L^p_{\alpha})$, $A_n \pto A$ in short, if
\[\lim\limits_{n \to \infty} \norm{(A_n-A)M_{\1_{B(0,r)}}} = \lim\limits_{n \to \infty} \norm{M_{\1_{B(0,r)}}(A_n-A)} = 0\]
for all $r > 0$. The connection between $*$-strong and $\Pc$-strong convergence is as follows:

\begin{prop} \label{prop:strong_convergence_P_convergence}
$\Pc$-strong convergence implies $*$-strong convergence. Moreover, if $A_n \starto A$, then $PA_nP \pto PAP$.
\end{prop}

\begin{proof}
The first statement is clear and the second follows again from Proposition \ref{prop:compact_convergence}.
\end{proof}

We also note that if $A_n \pto 0$ and $r > 0$, then
\[\norm{M_{\1_{\C^n \setminus B(x_0,r)}}A_nM_{\1_{\C^n \setminus B(x_0,r)}} - A_n} \leq \norm{M_{\1_{B(x_0,r)}}A_n} + \norm{M_{\1_{\C^n \setminus B(x_0,r)}}A_nM_{\1_{B(x_0,r)}}} \to 0\]
as $n \to \infty$. The following is a $\Pc$-theory generalization of \cite[Lemma 2.1]{MuXi}.

\begin{lem} \label{lem:P_theory_lemma}
Let $(K_n)_{n \in \N}$ be a bounded sequence of $\Pc$-compact operators such that $K_n \pto 0$. Then there is a strictly increasing sequence $(n_k)_{k \in \N}$ of positive integers such that $B := \sum\limits_{k = 1}^{\infty} K_{n_k}$ is a $\Pc$-strongly convergent series and
\[\norm{B}_{\Pc} = \limsup\limits_{k \to \infty} \norm{K_{n_k}}.\]
\end{lem}

\begin{proof}
Let $\epsilon > 0$. By $\Pc$-compactness, we have
\[\lim\limits_{r \to \infty} \norm{M_{\1_{B(x_0,r)}}K_nM_{\1_{B(x_0,r)}} - K_n} = 0\]
for all $n \in \N$. Hence, we may choose $r_1 > 0$ sufficiently large such that $A_1 := M_{\1_{B(x_0,r_1)}}K_1M_{\1_{B(x_0,r_1)}}$ satisfies $\norm{A_1-K_1} \leq \frac{\epsilon}{2}$. As $(K_n)_{n \in \N}$ converges $\Pc$-strongly to $0$, we have
\[\lim\limits_{n \to \infty} \norm{M_{\1_{\C^n \setminus B(x_0,r_1)}}K_nM_{\1_{\C^n \setminus B(x_0,r_1)}} - K_n} = 0.\]
Therefore, we can choose $n_2 > n_1 := 1$ such that
\[\norm{M_{\1_{\C^n \setminus B(x_0,r_1)}}K_{n_2}M_{\1_{\C^n \setminus B(x_0,r_1)}} - K_{n_2}} \leq \frac{\epsilon}{8}.\]
Next, by $\Pc$-compactness again, we can choose $r_2 > r_1$ such that
\[\norm{M_{\1_{B(x_0,r_2)}}M_{\1_{\C^n \setminus B(x_0,r_1)}}K_{n_2}M_{\1_{\C^n \setminus B(x_0,r_1)}}M_{\1_{B(x_0,r_2)}} - M_{\1_{\C^n \setminus B(x_0,r_1)}}K_{n_2}M_{\1_{\C^n \setminus B(x_0,r_1)}}} \leq \frac{\epsilon}{8},\]
which implies $\norm{A_2 - K_{n_2}} \leq \frac{\epsilon}{4}$ for $A_2 := M_{\1_{B(x_0,r_2) \setminus B(x_0,r_1)}}K_{n_2}M_{\1_{B(x_0,r_2) \setminus B(x_0,r_1)}}$. Interating this procedure, we get two strictly increasing sequences $(r_k)_{k \in \N}$ and $(n_k)_{k \in \N}$ and a sequence of operators $(A_k)_{k \in \N}$ such that
\[A_k = M_{\1_{B(x_0,r_k) \setminus B(x_0,r_{k-1})}}K_{n_k}M_{\1_{B(x_0,r_k) \setminus B(x_0,r_{k-1})}}\]
and $\norm{A_k - K_{n_k}} \leq \frac{\epsilon}{2^k}$ for all $k \in \N$. Define $A := \sum\limits_{k = 1}^{\infty} A_k$. As $A$ has block structure, it is easily seen that this series converges $\Pc$-strongly and
\[\norm{A} = \sup\limits_{k \in \N} \norm{A_k} \leq \sup\limits_{k \in \N} \norm{K_{n_k}} < \infty.\]
Moreover, $\norm{A}_{\Pc} \leq \limsup\limits_{k \to \infty} \norm{A_k}$ and
\[\norm{A+K} \geq \|(A+K)M_{\1_{\C^n \setminus B(0,r_{k-1})}}\| \geq \norm{A_k} - \|KM_{\1_{\C^n \setminus B(0,r_{k-1})}}\|\]
for all $\Pc$-compact operators $K$. Taking the limit $k \to \infty$, we get $\norm{A+K} \geq \limsup\limits_{k \to \infty} \norm{A_k}$. Therefore, the equality $\norm{A}_{\Pc} = \limsup\limits_{k \to \infty} \norm{A_k}$ follows. Now define $B := \sum\limits_{k = 1}^{\infty} K_{n_k}$. As $\norm{A_k - K_{n_k}} \leq \frac{\epsilon}{2^k}$ for all $k \in \N$, this series also converges $\Pc$-strongly and $\norm{B} \leq \norm{A} + \epsilon$. Moreover
\[\abs{\norm{B}_{\Pc} - \norm{A}_{\Pc}} \leq \norm{B-A}_{\Pc} \leq \norm{B-A} \leq \epsilon.\]
As $\epsilon$ was arbitrary, we get
\[\norm{B}_{\Pc} = \limsup\limits_{k \to \infty} \norm{A_k} = \limsup\limits_{k \to \infty} \norm{K_{n_k}}.\qedhere\]
\end{proof}

Let $\VO_b(\C^n)$ denote the set of bounded continuous functions $f \from \C^n \to \C$ of vanishing oscillation, that is,
\[\Osc_z(f) := \sup\set{\abs{f(z)-f(w)} : |w-z| \leq 1} \to 0\]
as $\abs{z} \to \infty$. Obviously, the set of continuous functions with compact support, $C_c(\C^n)$, is contained in $\VO_b(\C^n)$. For $t > 0$ and $j_0 \in \N$ we define the following sets of functions:
\[\mathscr{G}_{t,j_0} := \set{\sum\limits_{j = j_0+1}^k a_j\varphi_{j,t} : a_j \in \C, |a_j| = 1, k \in \N} \subset C_c(\C^n).\]
Note that $\norm{g}_{\infty} \leq N$ for all $g \in \mathscr{G}_{t,j_0}$, $t > 0$ and $j_0 \in \N$.

\begin{lem} \label{lem:P_theory_lemma_2}
Let $A \in \Lc(L^p)$ and assume that $[M_f,A]$ is $\Pc$-compact for all $f \in \VO_b(\C^n)$. Then for every $\epsilon > 0$, there is a $t > 0$ and a $j_0 \in \N$ such that for all $g \in \mathscr{G}_{j_0,t}$ we have $\norm{[M_g,A]} < \epsilon$.
\end{lem}

\begin{proof}
Assume that this is not the case, that is, there is an $\epsilon > 0$ such that for every $t > 0$ and every $j_0 \in \N$ there is an $g \in \mathscr{G}_{j_0,t}$ such that $\norm{[M_g,A]} \geq \epsilon$. Set $t_n = \frac{1}{n}$. As every ball $B(0,r)$ has a non-trivial intersection with only finitely many of the sets $\set{\supp \varphi_{j,t_n} : j \in \N}$, we can choose $j_n \in \N$ and $g_n \in \mathscr{G}_{j_n,t_n}$ recursively such that $\dist(\supp g_n,\supp g_m) > 1$ for $n \neq m$ and $\norm{[M_{g_n},A]} \geq \epsilon$ for all $n \in \N$. Let $K_n := [M_{g_n},A]$. As $g_n \in \mathscr{G}_{j_n,t_n} \subset C_c(\C^n) \subset \VO_b(\C^n)$, $K_n$ is $\Pc$-compact. Let $\psi \in C_c(\C^n)$. Then
\[M_{\psi}K_n = M_{\psi \cdot g_n}A - M_{\psi}AM_{g_n} = M_{\psi \cdot g_n}A - [M_{\psi},A]M_{g_n} + AM_{\psi \cdot g_n}.\]
By construction, the first and the third term are $0$ for sufficiently large $n$. As $[M_{\psi},A]$ is $\Pc$-compact by assumption, the second term also tends to $0$ as $n \to \infty$. Similarly, $K_nM_{\psi} \to 0$ as $n \to \infty$. It follows that $(K_n)_{n \in \N}$ converges $\Pc$-strongly to $0$. Lemma \ref{lem:P_theory_lemma} thus implies that there is a strictly increasing sequence $(n_k)_{k \in \N}$ such that
\begin{equation} \label{eq:P_theory_lemma_2}
\norm{\sum\limits_{k = 1}^{\infty} [M_{g_{n_k}},A]}_{\Pc} = \limsup\limits_{k \to \infty} \norm{[M_{g_{n_k}},A]} \geq \epsilon.
\end{equation}
Let $\frac{1}{p} + \frac{1}{q} = 1$, $z,w \in \C^n$ with $d(z,w) \leq 1$ and observe that
\[\abs{g_n(z)-g_n(w)} \leq \sum\limits_{j = j_n+1}^k \abs{\varphi_{j,t_n}(z) - \varphi_{j,t_n}(w)} \leq (2N)^{1/q}\left(\sum\limits_{j = j_n+1}^k \abs{\varphi_{j,t_n}(z) - \varphi_{j,t_n}(w)}^p\right)^{1/p}\]
by Hoelder's inequality and the fact that at most $2N$ of the summands can be non-zero. The latter is bounded by $(2N)^{1/q}t_n^{1/p} = \frac{(2N)^{1/q}}{n^{1/p}}$, which tends to $0$ as $n \to \infty$. It follows that $f := \sum\limits_{k = 1}^{\infty} g_{n_k}$ is in $\VO_b(\C^n)$. But this means that
\[0 = \norm{[M_f,A]}_{\Pc} = \norm{\sum\limits_{k = 1}^{\infty} [M_{g_{n_k}},A]}_{\Pc},\]
which contradicts \eqref{eq:P_theory_lemma_2}.
\end{proof}

The last result of this section is an elementary but useful lemma.

\begin{lem} \label{lem:sum_of_products}
Let $n \in \N$ and let $A_1, \ldots, A_n, B_1, \ldots B_n$ be bounded linear operators. Then
\[\Biggl\|\sum\limits_{j = 1}^n A_jB_j\Biggr\| \leq \Biggl\|\sum\limits_{j = 1}^n \zeta_jA_j\Biggr\| \cdot \Biggl\|\sum\limits_{j = 1}^n \overline{\zeta_j}B_j\Biggr\|,\]
where the $\zeta_j$ are certain roots of unity.
\end{lem}

\begin{proof}
Let $\xi_j = e^{\frac{2\pi ij}{n}}$ and $m \in \set{1-n, \ldots, n-1}$. Then
\[\sum\limits_{j = 1}^n \xi_j^m = \begin{cases} n & \text{if $m = 0$}, \\ 0 & \text{if $m \neq 0$}. \end{cases}\]
Since
\[\sum\limits_{j = 1}^n \sum\limits_{k = 1}^n \sum\limits_{l = 1}^n A_jB_k\xi_l^{j-k} = n\sum\limits_{j = 1}^n A_jB_j,\]
there is an $l \in \set{1, \ldots, n}$ such that
\[\Biggl\|\sum\limits_{j = 1}^n \sum\limits_{k = 1}^n A_jB_k\xi_l^{j-k}\Biggr\| \geq \Biggl\|\sum\limits_{j = 1}^n A_jB_j\Biggr\|.\]
Defining $\zeta_j := \xi_l^j$ yields the result.
\end{proof}

\section{Essential Commutants} \label{sec:Essential_Commutants}

We now return to the Fock space. For a set of operators $X \subseteq \Lc(F^p_{\alpha})$ we call the set
\[\EssCom(X) := \set{A \in \Lc(F^p_{\alpha}) : [B,A] \in \Kc(F^p_{\alpha}) \text{ for all } B \in X}\]
the essential commutant of $X$. The set $\EssCom(\EssCom(X))$ is called the essential bicommutant of $X$. Of course, we always have $X \subseteq \EssCom(\EssCom(X))$. As we will see in this section, equality holds for $X = \Ac^p$. We start with a corollary to Lemma \ref{lem:P_theory_lemma_2}. This corollary will be crucial for the upcoming theorem.

\begin{cor} \label{cor:P_theory_lemma_2_cor}
Let $A \in \Lc(F^p_{\alpha})$. If $[T_f,A]$ is compact for all $f \in \VO_b(\C^n)$, then for every $\varepsilon > 0$, there is a $t > 0$ and a $j_0 \in \N$ such that for all $g \in \mathscr{G}_{j_0,t}$ we have $\norm{[M_g,AP]} < \epsilon$.
\end{cor}

For the proof, we need the notion of a Hankel operator $H_f = (I-P)M_f \from F^p_{\alpha} \to L^p_{\alpha}$, where $I$ denotes the identity operator and $f$ is a bounded symbol. We will also need the adjoint of $H_{\bar f} \from F^q_{\alpha} \to L^q_{\alpha}$, which is given by $H_{\bar f}^* = PM_f(I-P) \from L^p_{\alpha} \to F^p_{\alpha}$, where $\frac{1}{p} + \frac{1}{q} = 1$.

\begin{proof}
In view of Lemma \ref{lem:P_theory_lemma_2}, we need to check that $[M_f,AP]$ is $\Pc$-compact for all $f \in \VO_b(\C^n)$:
\begin{align*}
[M_f,AP] &= M_fAP - APM_f = PM_fAP + (I-P)M_fAP - APM_fP - APM_f(I-P)\\
&= [T_f,A]P + (I-P)M_fAP - APM_f(I-P) = [T_f,A]P + H_fAP - AH_{\bar f}^*.
\end{align*}
As $H_f$ and $H_{\bar f}$ are compact for $f \in \VO_b(\C^n)$ (see \cite[Theorem 1.1]{Lv}, for example), $[M_f,AP]$ is compact, hence $\Pc$-compact (Proposition \ref{prop:compact_P_compact}). Therefore, Lemma \ref{lem:P_theory_lemma_2} implies the result.
\end{proof}

\begin{thm} \label{thm:compact_commutator_implies_BDO}
Let $A \in \Lc(F^p_{\alpha})$. If $[T_f,A]$ is compact for all $f \in \VO_b(\C^n)$, then $A \in \Ac^p$.
\end{thm}

\begin{proof}
Let $A \in \Lc(F^p_{\alpha})$ and $\frac{1}{p} + \frac{1}{q} = 1$. We decompose $AP$ as
\[AP = \sum\limits_{j = 1}^{\infty} M_{\varphi_{j,t}^p}AP = \sum\limits_{j = 1}^{\infty} M_{\varphi_{j,t}^{p/q}}APM_{\varphi_{j,t}} + \sum\limits_{j = 1}^{\infty} M_{\varphi_{j,t}^{p/q}}[M_{\varphi_{j,t}},AP].\]
The first summand is clearly a band operator. The second summand can be further decomposed as
\[\sum\limits_{j = 1}^{\infty} M_{\varphi_{j,t}^{p/q}}[M_{\varphi_{j,t}},AP] = \sum\limits_{j = 1}^{j_0} M_{\varphi_{j,t}^{p/q}}[M_{\varphi_{j,t}},AP] + \sum\limits_{j = j_0+1}^{\infty} M_{\varphi_{j,t}^{p/q}}[M_{\varphi_{j,t}},AP]\]
for any $j_0 \in \N$. As $\varphi_{j,t}$ has compact support, $[M_{\varphi_{j,t}},AP]$ is compact, hence band-dominated as well. It therefore suffices to show, by choosing $t$ and $j_0$ appropriately, that $\sum\limits_{j = j_0+1}^{\infty} M_{\varphi_{j,t}^{p/q}}[M_{\varphi_{j,t}},AP]$ can be made arbitrarily small. The theorem then follows from the fact that $\BDO^p$ is closed. 

Let $\varepsilon > 0$. We have
\[\norm{\sum\limits_{j = j_0+1}^k M_{\varphi_{j,t}^{p/q}}[M_{\varphi_{j,t}},AP]} \leq \norm{\sum\limits_{j = j_0+1}^k \zeta_jM_{\varphi_{j,t}^{p/q}}}\norm{\sum\limits_{j = j_0+1}^k \overline{\zeta_j}[M_{\varphi_{j,t}},AP]}\]
for every $k \in \N$, where the $\zeta_j$ are the appropriate roots of unity according to Lemma \ref{lem:sum_of_products}. As $\set{j \in \N : \varphi_{j,t}(z) \neq 0}$ has at most $N$ elements for each $z \in \C^n$, the first factor is bounded by the universal constant $N$. For the second factor we observe
\[\sum\limits_{j = j_0+1}^k \overline{\zeta_j}[M_{\varphi_{j,t}},AP] = [M_g,AP],\]
where $g = \sum\limits_{j = j_0+1}^k \overline{\zeta_j}\varphi_{j,t} \in \mathscr{G}_{j_0,t}$. By choosing $t$ and $j_0$ appropriately, this can be bounded by $\epsilon$ (Corollary \ref{cor:P_theory_lemma_2_cor}). As this estimate is independent of $g \in \mathscr{G}_{j_0,t}$, it is also independent of $k$. Therefore, using that the series is strongly convergent, we obtain
\[\norm{\sum\limits_{j = j_0+1}^{\infty} M_{\varphi_{j,t}^{p/q}}[M_{\varphi_{j,t}},AP]} \leq \liminf_{k \to \infty} \norm{\sum\limits_{j = j_0+1}^k M_{\varphi_{j,t}^{p/q}}[M_{\varphi_{j,t}},AP]} \leq N\epsilon.\qedhere\]
\end{proof}

\begin{prop} \label{prop:sequential_convergence}~
\begin{itemize}
	\item[(a)] Let $A \in \Ac^p$ and $x \in \beta\C^n \setminus \C^n$. Then there is a sequence $(x_n)_{n \in \N}$ in $\C^n$ with $\lim\limits_{n \to \infty} \abs{x_n} = \infty$ such that $W_{-x_n}AW_{x_n} \to A_x$ in the strong operator topology.
	\item[(b)] Let $A \in \Ac^p$ and $(x_n)_{n \in \N}$ a sequence in $\C^n$ with $\lim\limits_{n \to \infty} \abs{x_n} = \infty$. Then there exist a subsequence $(x_{n_k})_{k \in \N}$ and an operator $B \in \Ac^p$ such that $W_{-x_{n_k}}BW_{x_{n_k}} \to A$ in the strong operator topology.
\end{itemize}
\end{prop}

\begin{proof}
(a) This follows from the fact that the strong operator topology is metrizable on bounded subsets of $\Lc(F^p_{\alpha})$ (see Remark 5.2 in \cite{HaSe}).

(b) Choose a subsequence $(x_{n_k})_{k \in \N}$ of $(x_n)_{n \in \N}$ such that $|x_{n_{k+1}}| > |x_{n_k}| + 2k+1$. This ensures that $B(x_{n_k},k) \cap B(x_{n_l},l) = \emptyset$ for $k \neq l$. Define
\[C_l := M_{\1_{B(x_{n_l},l)}}W_{x_{n_l}}APW_{-x_{n_l}}M_{\1_{B(x_{n_l},l)}}\]
and $C := \sum\limits_{l = 1}^{\infty} C_l$. Considering the block structure of $C$, it is clear that $C$ is bounded with $\norm{C} \leq \norm{AP}$. By approximation, this also shows that $C \in \BDO^p$. We will now show that $(W_{-x_{n_k}}CW_{x_{n_k}})_{k \in \N}$ converges strongly to $AP$. Restricting to $F^p_{\alpha}$, $B := PC|_{F^p_{\alpha}} \in \Ac^p$, then yields the assertion.

Let $j \in \N$ and $f \in L^p_{\alpha}$ with $\supp f \subseteq B(0,j)$. For $k \geq j$ the support of $W_{x_{n_k}}f$ is contained in $B(x_{n_k},j) \subseteq B(x_{n_k},k)$ and therefore, by construction, $W_{-x_{n_k}}C_lW_{x_{n_k}}f = 0$ unless $l = k$. For $l = k$ we have
\[W_{-x_{n_k}}C_lW_{x_{n_k}} = W_{-x_{n_k}}M_{\1_{B(x_{n_k},k)}}W_{x_{n_k}}APW_{-x_{n_k}}M_{\1_{B(x_{n_k},k)}}W_{x_{n_k}} = M_{\1_{B(0,k)}}APM_{\1_{B(0,k)}}.\]
It follows
\[W_{-x_{n_k}}CW_{x_{n_k}}f = W_{-x_{n_k}}C_kW_{x_{n_k}}f = M_{\1_{B(0,k)}}APf.\]
As $M_{\1_{B(0,k)}} \to I$ strongly, $(W_{-x_{n_k}}CW_{x_{n_k}}f)_{k \in \N}$ converges to $APf$. As functions with compact support are dense in $L^p_{\alpha}$, we conclude that $(W_{-x_{n_k}}CW_{x_{n_k}})_{k \in \N}$ converges strongly to $AP$. 
\end{proof}

\begin{thm} \label{thm:EssComBDO}
Let $A \in \Lc(F^p_{\alpha})$. Then the following are equivalent:
\begin{itemize}
	\item[(a)] $A \in \EssCom(\Ac^p)$.
	\item[(b)] $A \in \Ac^p$ and $A_x \in \C I := \set{\lambda I : \lambda \in \C}$ for all $x \in \beta\C^n \setminus \C^n$.
	\item[(c)] $A = T_f + K$ for some $f \in \VO_b(\C^n)$ and $K \in \Kc(F^p_{\alpha})$.
\end{itemize}
\end{thm}

\begin{proof}
(a) $\Rightarrow$ (b): Let $A \in \Lc(F^p_{\alpha})$ and assume that $[B,A]$ is compact for all $B \in \Ac^p$. As $T_f \in \Ac^p$ for all $f \in L^{\infty}(\C^n)$, Theorem \ref{thm:compact_commutator_implies_BDO} implies that $A \in \Ac^p$. Let $x \in \beta\C^n \setminus \C^n$ and $C \in \Ac^p$. By Proposition \ref{prop:sequential_convergence}, we can choose a sequence $(x_n)_{n \in \N}$ with $\lim\limits \abs{x_n} = \infty$ such that
\[A_x = \slim\limits_{n \to \infty} W_{-x_n}AW_{x_n} \quad \text{and} \quad C = \slim\limits_{n \to \infty} W_{-x_n}BW_{x_n}\]
for some $B \in \Ac^p$. As $[B,A]$ is compact, it follows
\[0 = \slim\limits_{n \to \infty} W_{-x_n}[B,A]W_{x_n} = CA_x - A_xC,\]
that is, $A_x$ commutes with $C$. As $C \in \Ac^p$ was arbitrary and $\Kc(F^p_{\alpha}) \subset \Ac^p$, this means that $A_x$ commutes with every compact operator. That this can only happen if $A_x$ is a multiple of the identity is a well-known fact.

(b) $\Rightarrow$ (a): Conversely, if every limit operator of $A$ is a constant multiple of the identity, then
\[[B,A]_x = B_xA_x - A_xB_x = 0\]
for all $x \in \beta\C^n \setminus \C^n$ and $B \in \Ac^p$. The characterization of compact operators (Proposition \ref{prop:limit_operator_characterization}) thus implies $A \in \EssCom(\Ac^p)$.

(c) $\Rightarrow$ (b): As Toeplitz and compact operators are in $\Ac^p$ (Proposition \ref{prop:A^p_properties}) and every limit operator of a compact operator is $0$ (Proposition \ref{prop:limit_operator_characterization}), it suffices to check $(T_f)_x = f(x)I$ for all $x \in \beta\C^n \setminus \C^n$, $f \in \VO_b(\C^n)$. Using $W_{-z}T_fW_z = T_{f(\cdot +z)}$, this is easily verified (as in the proof of \cite[Theorem 33]{FuHa}).

(b) $\Rightarrow$ (c): We will first show that $\Bc(A)$ is in $\VO_b(\C^n)$. The argument is similar to \cite[Theorem 36]{HaggerBSD}. For completeness we sketch it here. Assume that $\Bc(A) \not\in \VO_b(\C^n)$. Then we can find an $\epsilon > 0$ and nets $(x_{\gamma})$, $(h_{\gamma})$ in $\C^n$ with $|h_{\gamma}| \leq 1$ such that $x_{\gamma} \to x$ for some $x \in \beta\C^n \setminus \C^n$, $h_{\gamma} \to h$ for some $h \in \C^n$ and 
\begin{equation} \label{eq:EssComBDO}
\abs{[\Bc(A)](x_{\gamma}+h_{\gamma}) - [\Bc(A)](x_{\gamma})} > \epsilon
\end{equation}
for all $\gamma$. By Proposition \ref{prop:limit_operators}, the net $(W_{-x_{\gamma}}AW_{x_{\gamma}})$ converges strongly to $A_x = \lambda_x I$ for some $\lambda_x \in \C$. Since $W_{x_{\gamma}}k_z = e^{-i\alpha\Imag(x_{\gamma} \cdot \bar{z})}k_{x_{\gamma}+z}$ for $z \in \C^n$, we have
\begin{equation} \label{eq:EssComBDO_2}
[\Bc(A)](x_{\gamma}+z) = [\Bc(W_{-x_{\gamma}}AW_{x_{\gamma}})](z) \to [\Bc(\lambda_x I)](z) = \lambda_x
\end{equation}
as $x_{\gamma} \to x$. It follows
\begin{align*}
&\abs{[\Bc(A)](x_{\gamma} + h_{\gamma}) - [\Bc(A)](x_{\gamma})}\\
&\qquad \qquad \qquad \qquad \qquad \leq \abs{[\Bc(A)](x_{\gamma} + h_{\gamma}) - [\Bc(A)](x_{\gamma} + h)} + \abs{[\Bc(A)](x_{\gamma} + h) - [\Bc(A)](x_{\gamma})} \to 0
\end{align*}
because of the Lipschitz continuity of $\Bc(A)$ (see Proposition \ref{prop:Berezin_Lipschitz}) and \eqref{eq:EssComBDO_2}. This is a contradiction to \eqref{eq:EssComBDO}. Hence $\Bc(A) \in \VO_b(\C^n)$. It remains to show that $A - T_{\Bc(A)}$ is compact. However, the computation \eqref{eq:EssComBDO_2} for $z = 0$ also shows that $A_x = [B(A)](x)I = (T_{\Bc(A)})_x$ (compare with the third part of this proof) for all $x \in \beta\C^n \setminus \C^n$. Using the compactness characterization again (Proposition \ref{prop:limit_operator_characterization}), we conclude that $A = T_{\Bc(A)} + K$ for some compact operator $K$.
\end{proof}

Combining Theorem \ref{thm:compact_commutator_implies_BDO} with Theorem \ref{thm:EssComBDO}, we get the following corollary:

\begin{cor} \label{cor:DoubleEssComBDO}
$\Ac^p$ is equal to its essential bicommutant.
\end{cor}

\section{Characterizations of the Toeplitz algebra} \label{sec:Characterizations}

Recall that an operator $A \in \Lc(F^p_{\alpha})$ is called sufficiently localized if there are constants $C > 0$, $\beta > 2n$ such that
\[\sp{Ak_z}{k_w} \leq \frac{C}{(1+|z-w|)^{\beta}}\]
for all $w,z \in \C^n$. As
\begin{align} \label{eq:Integral_operator}
(Af)(w)e^{-\frac{\alpha}{2}|w|^2} &= \left(\frac{\alpha}{\pi}\right)^n \int_{\C^n} f(z) \overline{(A^*k_w)(z)} e^{-\alpha|z|^2} \, \mathrm{d}z\notag\\
&= \left(\frac{\alpha}{\pi}\right)^n \int_{\C^n} f(z)\sp{Ak_z}{k_w} e^{-\frac{\alpha}{2}|z|^2} \, \mathrm{d}z,
\end{align}
Young's inequality implies
\[\norm{A}_p := \sup\limits_{f \in F^p_{\alpha}} \frac{\norm{Af}_p}{\norm{f}_p} \leq C\left(\frac{\alpha}{\pi}\right)^n \norm{\frac{1}{(1 + |\cdot|)^{\beta}}}_{L^1(\C^n)}\]
for every $p \in (1,\infty)$. In particular, \eqref{eq:Integral_operator} defines a bounded linear operator on every $F^p_{\alpha}$. Now, before we can summarize this paper in our main result, we need one last proposition.

\begin{prop} \label{prop:suff_localized}
Let $A$ be sufficiently localized. Then the map $\Psi \from \C^n \to \Lc(F^p_{\alpha})$, $\Psi(z) = W_{-z}AW_z$ is uniformly continuous in the norm topology.
\end{prop}

\begin{proof}
Let $A$ be sufficiently localized. Because of the identities $W_{z+x} = e^{-i\alpha\Imag(z \cdot \overline{x})}W_xW_z$ and $W_{-(z+x)} = e^{i\alpha\Imag(z \cdot \overline{x})}W_{-z}W_{-x}$, it suffices to check the continuity at $z = 0$. By \eqref{eq:Integral_operator}, we have
\begin{align*}
(Af)(w) &= \left(\frac{\alpha}{\pi}\right)^n \int_{\C^n} f(z)\sp{Ak_z}{k_w} e^{\frac{\alpha}{2}(|w|^2 - |z|^2)} \, \mathrm{d}z\\
&= \left(\frac{\alpha}{\pi}\right)^n \int_{\C^n} f(w-z)\sp{Ak_{w-z}}{k_w} e^{\frac{\alpha}{2}(|w|^2 - |w-z|^2)} \, \mathrm{d}z\\
&= \left(\frac{\alpha}{\pi}\right)^n \int_{\C^n} (W_zf)(w) \sp{Ak_{w-z}}{k_w} e^{\frac{\alpha}{2}(|w|^2 - |w-z|^2 - 2w \cdot \overline{z} + |z|^2)} \, \mathrm{d}z\\
&= \left(\frac{\alpha}{\pi}\right)^n \int_{\C^n} (W_zf)(w) \sp{Ak_{w-z}}{k_w} e^{i\alpha\Imag(z \cdot \overline{w})} \, \mathrm{d}z\\
&= \left(\frac{\alpha}{\pi}\right)^n \int_{\C^n} (W_zf)(w) \sp{AW_{-z}k_w}{k_w} \, \mathrm{d}z\\
&= \left(\frac{\alpha}{\pi}\right)^n \int_{\C^n} (W_zf)(w) [\Bc(AW_{-z})](w) \, \mathrm{d}z.
\end{align*}
Define $g_z(w) := \left(\frac{\alpha}{\pi}\right)^n [\Bc(AW_{-z})](w)$. By Proposition \ref{prop:Berezin_Lipschitz}, the functions $g_z$ are Lipschitz continuous with Lipschitz constants bounded by $\tilde{C}\sqrt{\alpha}\left(\frac{\alpha}{\pi}\right)^n\norm{A}$ for some constant $\tilde{C} > 0$. Moreover, we have $\norm{g_z}_{\infty} \leq \left(\frac{\alpha}{\pi}\right)^n \frac{C}{(1+|z|)^{\beta}}$, that is, the function $z \mapsto \norm{g_z}_{\infty}$ is in $L^1(\C^n)$. It follows that we can write
\[A = \int_{\C^n} M_{g_z}W_z \, \mathrm{d}z\]
and the integral makes sense pointwise as either a Riemann or a Bochner integral on $F^p_{\alpha}$. Therefore,
\begin{align*}
\norm{W_{-x}AW_x - A} &\leq \int_{\C^n} \norm{W_{-x}M_{g_z}W_zW_x - M_{g_z}W_z} \, \mathrm{d}z\\
&= \int_{\C^n} \norm{W_{-x}M_{g_z}W_xW_{-x}W_zW_x - M_{g_z}W_z} \, \mathrm{d}z\\
&= \int_{\C^n} \norm{e^{-2i\alpha\Imag(z \cdot \bar{x})}M_{g_z(\cdot+x)}W_z - M_{g_z}W_z} \, \mathrm{d}z\\
&\leq \int_{\C^n} \norm{e^{-2i\alpha\Imag(z \cdot \bar{x})}(M_{g_z(\cdot+x)} - M_{g_z})} \, \mathrm{d}z + \int_{\C^n} \norm{(e^{-2i\alpha\Imag(z \cdot \bar{x})}-1)M_{g_z}} \, \mathrm{d}z
\end{align*}
for every $x \in \C^n$. By dominated convergence, it suffices to check that the integrands converge pointwise to $0$. For the first term we have
\[\norm{e^{-2i\alpha\Imag(z \cdot \bar{x})}(M_{g_z(\cdot+x)} - M_{g_z})} \leq \sup\limits_{w \in \C^n} |g_z(w+x)-g_z(w)| \leq \tilde{C}\sqrt{\alpha}\left(\frac{\alpha}{\pi}\right)^n\norm{A}|x|\]
by the Lipschitz continuity of $g_z$. For the second term we get
\[\norm{(e^{-2i\alpha\Imag(z \cdot \bar{x})}-1)M_{g_z}} \leq \left(\frac{\alpha}{\pi}\right)^n \frac{2C}{(1+|z|)^{\beta}}\abs{\sin\bigl{(}\alpha\Imag(z \cdot \bar{x})\bigr{)}}.\]
Hence, $\norm{W_{-x}AW_x - A} \to 0$ as $x \to 0$.
\end{proof}

We can now prove the main result of this paper, which shows that, after taking the closure, all the previously introduced algebras are actually the same. We use the abbreviation $\BUC(\C^n)$ for the set of bouded and uniformly continuous functions on $\C^n$.

\begin{thm} \label{thm:all_algebras_coincide}
Let $A \in \Lc(F^p_{\alpha})$. The following are equivalent:
\begin{itemize}
	\item[(a)] $A$ is contained in the closure of $\set{T_f : f \in \BUC(\C^n)}$.
	\item[(b)] $A$ is contained in the closure of $\set{T_f : f \in L^{\infty}(\C^n)}$.
	\item[(c)] $A \in \Tc^p$.
	\item[(d)] $A \in \Ac^p$.
	\item[(e)] $A$ is contained in the closure of $\Ac_{sl}$
	\item[(f)] $A$ is contained in the closure of $\Ac_{wl}$
	\item[(g)] The map $\Psi: z \mapsto W_{-z}AW_z$ is uniformly continuous in the norm topology.
	\item[(h)] $[T_f,A]$ is compact for all $f \in \VO_b(\C^n)$.
\end{itemize}
\end{thm}

\begin{proof}
As all the work has been done above or in previous work, we only need to connect the statements by arrows. ``(a) $\Longrightarrow$ (b) $\Longrightarrow$ (c)'' is clear. ``(c) $\Longrightarrow$ (d)'' is \cite[Theorem 15]{FuHa}. ``(d) $\Longrightarrow$ (e)'' follows as in \cite[Theorem 4.20]{BaFu}. ``(e) $\Longrightarrow$ (f)'' is again clear and ``(f) $\Longrightarrow$ (d)'' follows from \cite[Proposition 3.5]{IsMiWi}. ``(e) $\Longrightarrow$ (g)'' is shown in Proposition \ref{prop:suff_localized} above. ``(g) $\Longrightarrow$ (a)'' was shown in \cite[Theorem 3.1]{Fulsche}. The equivalence ``(d) $\Longleftrightarrow$ (h)'' follows from Theorem \ref{thm:compact_commutator_implies_BDO} and Theorem \ref{thm:EssComBDO}.
\end{proof}

Of course, Theorem \ref{thm:all_algebras_coincide} has many direct corollaries as many results have been proven for these classes of operators. As they turn out to be all the same, many results directly carry over. Let us highlight just one of them here, which follows from Proposition \ref{prop:A^p_properties}. Alternatively, one can also deduce it directly from Theorem \ref{thm:all_algebras_coincide} (h). If $p = 2$, the inverse closedness is of course trivial as these algebras are $C^*$ in that case.

\begin{cor}
$\overline{\Ac_{sl}}$, $\overline{\Ac_{wl}}$ and $\Tc^p$ are inverse closed Banach algebras.
\end{cor}

Theorem \ref{thm:all_algebras_coincide} can also be used to give a different proof of ``(c) implies (a)'' in Theorem \ref{thm:EssComBDO}. Indeed, for $f,g \in L^{\infty}(\C^n)$ the following algebraic identity holds:
\[[T_f,T_g] = H_{\bar g}^*H_f - H_{\bar f}^*H_g,\]
where $H_{\bar f}^*$ is the adjoint of $H_{\bar f} \from F^q_{\alpha} \to L^q_{\alpha}$ as before. Now if $f \in \VO_b(\C^n)$, then $H_f$ and $H_{\bar f}$ are compact (see \cite[Theorem 1.1]{Lv}, for example), so that $[T_f,T_g]$ is compact as well. Hence, $T_f$ essentially commutes with all Toeplitz operators and therefore, by standard properties of the commutator, $T_f \in \EssCom(\Tc^p)$.

In fact, the boundedness of $g$ is not necessary in the above. We only needed the boundedness of $H_g$ and $H_{\bar g}$. As this is exactly the case when $g \in \BMO(\C^n)$ (see \cite[Theorem 1.1]{Lv}, for example), we get the following corollary of Theorem \ref{thm:compact_commutator_implies_BDO} and Theorem \ref{thm:all_algebras_coincide}.

\begin{cor}
If $g \in \BMO(\C^n)$ and $T_g \in \Lc(F^p_{\alpha})$, then $T_g$ is contained in the Toeplitz algebra.
\end{cor}

We note that it can also be verified directly that $T_g$ is sufficiently localized (similarly as it was done  for $p = 2$ in \cite[Lemma 4.11]{BaFu}).

\bigskip

\noindent
Raffael Hagger\\
Department of Mathematics and Statistics\\
University of Reading\\
Whiteknights Campus\\
Reading RG6 6AX\\
UNITED KINGDOM\\
r.t.hagger@reading.ac.uk

\end{document}